\documentclass[a4paper,12pt]{article}

\usepackage{amscd}
\usepackage{amssymb,amsfonts,amsmath,amsthm}
\usepackage{verbatim}

\usepackage{amsmath}
\usepackage{amsthm}
\usepackage{amssymb}

\usepackage{latexsym}
\usepackage{array}
\usepackage{enumerate}

\numberwithin{equation}{section}

\usepackage{amsmath,amssymb,mathrsfs,amsthm}
\usepackage[all]{xy}
\usepackage{graphicx}
\usepackage{ulem}
\usepackage{ascmac}
\usepackage{mathtools}

\newcommand{\BDC}{{\mathbf{D}}^{\mathrm{b}}}

\newcommand{\Mod}{\mathrm{Mod}}
\newcommand{\Hom}{\mathrm{Hom}}

\newcommand{\CC}{\mathbb{C}}

\newcommand{\RR}{\mathbb{R}}

\newcommand{\ZZ}{\mathbb{Z}}
\newcommand{\D}{\mathcal{D}}

\newcommand{\F}{\mathcal{F}}

\newcommand{\M}{\mathcal{M}}
\newcommand{\N}{\mathcal{N}}

\newcommand{\an}{{\rm an}}

\newcommand{\Sol}{{\rm Sol}}

\newcommand{\tl}[1]{\widetilde{#1}}

\newcommand{\simto}{\overset{\sim}{\longrightarrow}}

\newcommand{\op}{{\mbox{\scriptsize op}}}
\newcommand{\SD}{\mathcal{D}}

\newcommand{\SO}{\mathcal{O}}

\newcommand{\SM}{\mathcal{M}}

\newcommand{\SF}{\mathcal{F}}

\newcommand{\SH}{\mathcal{H}}
\newcommand{\calI}{\mathcal{I}}

\newcommand{\Conn}{\mathrm{Conn}}

\newcommand{\Modhol}{\mathrm{Mod}_{\mbox{\rm \scriptsize hol}}}
\newcommand{\Modrh}{\mathrm{Mod}_{\mbox{\rm \scriptsize rh}}}

\newcommand{\BDCcoh}{{\mathbf{D}}^{\mathrm{b}}_{\mbox{\rm \scriptsize coh}}}

\newcommand{\BDChol}{{\mathbf{D}}^{\mathrm{b}}_{\mbox{\rm \scriptsize hol}}}
\newcommand{\BDCrh}{{\mathbf{D}}^{\mathrm{b}}_{\mbox{\rm \scriptsize rh}}}
\newcommand{\BDCmero}{{\mathbf{D}}^{\mathrm{b}}_{\mbox{\rm \scriptsize mero}}}

\newcommand{\Dotimes}{\overset{D}{\otimes}}

\newcommand{\Potimes}{\overset{+}{\otimes}}

\newcommand{\rhom}{{\bfR}{\mathcal{H}}om}
\newcommand{\rihom}{{\bfR}{\mathcal{I}}hom}
\newcommand{\Prihom}{{\bfR}{\mathcal{I}}hom^+}

\newcommand{\I}{{\rm I}}
\newcommand{\che}[1]{\check{#1}}
\newcommand{\var}[1]{\overline{#1}}
\newcommand{\BEC}{{\mathbf{E}}^{\mathrm{b}}}

\newcommand{\ZEC}{{\mathbf{E}}^{\mathrm{0}}}
\newcommand{\ZECmero}{{\mathbf{E}}^{\mathrm{0}}_{\mbox{\rm \scriptsize mero}}}
\newcommand{\BECmero}{{\mathbf{E}}^{\mathrm{b}}_{\mbox{\rm \scriptsize mero}}}
\newcommand{\Q}{\mathbf{Q}}

\newcommand{\EE}{\mathbb{E}}
 
\newcommand{\bs}{\backslash}

\newcommand{\bfR}{\mathbf{R}}
\newcommand{\bfL}{\mathbf{L}}
\newcommand{\bfD}{\mathbf{D}}

\newcommand{\rmR}{{\rm R}}
\newcommand{\rmE}{{\rm E}}

\newcommand{\bfE}{\mathbf{E}}

\renewcommand{\Re}{\operatorname{Re}}

\newcommand{\reg}{{\rm reg}}
\newcommand{\sh}{{\rm sh}}

\newcommand{\ord}{\operatorname{ord}}

\newcommand{\AC}{(\mathbf{AC})}
\newcommand{\rami}{{\rm rm}}
\newcommand{\modi}{{\rm md}}
\newcommand{\blow}{{\rm bl}}

\newtheorem{theorem}{Theorem}[section]

\newtheorem{lemma}[theorem]{Lemma}

\newtheorem{proposition}[theorem]{Proposition}
\newtheorem{fact}[theorem]{Fact}

\theoremstyle{definition}
\newtheorem{definition}[theorem]{Definition}
\theoremstyle{remark}
\newtheorem{remark}[theorem]{\sc Remark}

\title{Another proof of the Riemann--Hilbert Correspondence\\ for Regular Holonomic $\D$-Modules}

\author{Yohei ITO\footnote{Department of Mathematics, Faculty of Science Division II, Tokyo University of Science, 1-3, Kagurazaka, Shinjuku-ku, Tokyo, 162-8601, Japan. E-mail: yitoh@rs.tus.ac.jp }}

\date{}

\begin{document}
\maketitle

\begin{abstract}
In this paper, we reprove
the Riemann--Hilbert correspondence for regular holonomic $\D$-modules of \cite{Kas84}
(see also \cite{Meb84})
by using the irregular Riemann--Hilbert correspondence of \cite{DK16}.
Moreover, we also prove the algebraic one by the same argument.
For this purpose, we study $\CC$-constructible enhanced ind-sheaves of \cite{Ito20, Ito21} in more detail.
\end{abstract}

\section{Introduction}
In 1984, the Riemann-Hilbert correspondence for analytic regular holonomic $\D$-modules
was established by M.\:Kashiwara \cite{Kas84}
as the equivalence of categories below (see also \cite{Meb84}).
Let $X$ be a complex manifold.
We denote by $\BDCrh(\D_X)$ the triangulated category of regular holonomic $\D_X$-modules,
by $\BDC_{\CC\mbox{\scriptsize -}c}(\CC_X)$ the one of $\CC$-constructible sheaves on X
and by $\Sol_X$ the solution functor.

\begin{fact}[{\cite[Main Theorem]{Kas84}, see also \cite[Thm.\:2.1.1]{Meb84}}]\label{regRH_ana}
There exists an equivalence of triangulated categories
$$\Sol_X\colon  \BDCrh(\D_X)^{\op}\simto\BDC_{\CC\mbox{\scriptsize -}c}(\CC_X).$$
\end{fact}

After the appearance of Fact \ref{regRH_ana},
A.\:Beilinson and J.\:Bernstein developed systematically
a theory of regular holonomic $\D$-modules on smooth algebraic varieties
over the complex number field $\CC$ and
obtained an algebraic version of Fact \ref{regRH_ana} as follows.
Let $X$ be a smooth algebraic variety over $\CC$
and denote by $X^\an$ the underlying complex manifold of $X$.
We denote by $\BDCrh(\D_X)$ the triangulated category of regular holonomic $\D_X$-modules,
by $\BDC_{\CC\mbox{\scriptsize -}c}(\CC_X)$ the one of algebraic $\CC$-constructible sheaves on $X^\an$.

\begin{fact}
[{\cite[Main Theorem C (c)]{Be} and \cite[Theorem 14.4]{Bor87}, see also \cite[\S 4]{Sai89}}]\label{regRH_alg}
There exists an equivalence of categories
\[\Sol_X \colon \BDCrh(\D_X)^{\op}\simto\BDC_{\CC\mbox{\scriptsize -}c}(\CC_X),\
\M \mapsto 
\Sol_X(\M) := \Sol_{X^\an}(\M^\an).\]
\end{fact}

The problem of extending
the Riemann--Hilbert correspondence to cover
the case of holonomic $\D$-modules with irregular singularities
had been open for 30 years.
After a groundbreaking development in the theory of irregular meromorphic connections by 
K.\:S.\:Kedlaya \cite{Ked10, Ked11} and T.\:Mochizuki \cite{Mochi09, Mochi11},
A.\:D'Agnolo and M.\:Kashiwara established the Riemann--Hilbert correspondence
for analytic irregular holonomic $\D$-modules in \cite{DK16} as follows.
Let $X$ be a complex manifold.
We denote by $\BDChol(\D_{X})$ the triangulated category of
holonomic $\D_X$-modules and by $\BEC_{\RR\mbox{\scriptsize -}c}(\I\CC_X)$
the one of $\RR$-constructible enhanced ind-sheaves on $X$.

\begin{fact}[{\cite[Thm.\:9.5.3]{DK16}}]\label{irregRH_DK}
There exists a fully faithful embedding
\[\Sol_X^{\rmE} \colon \BDChol(\D_X)^{\op}\hookrightarrow\BEC_{\RR\mbox{\scriptsize -}c}(\I\CC_X).\]
\end{fact}

Furthermore, T.\:Mochizuki proved that
the essential image of $\Sol_X^{\rmE}$ can be characterized  by the curve test \cite{Mochi16}.
On the other hand, 
in \cite[Thm.\:6.2]{Kas16},
M.\:Kashiwara showed the similar result of Fact \ref{irregRH_DK}
by using enhanced subanalytic sheaves instead of enhanced ind-sheaves,
see also \cite{Ito21b}.
In \cite[Thm.\:8.6]{Kuwa18},
T.\:Kuwagaki introduced another approach
to the irregular Riemann--Hilbert correspondence
via irregular constructible sheaves which are defined by $\CC$-constructible sheaves
with coefficients in a finite version of the Novikov ring and special gradings.

In \cite{Ito20}, the author defined $\CC$-constructibility for enhanced ind-sheaves on a complex manifold $X$
and proved that they are nothing but objects of the essential image of $\Sol_X^{\rmE}$.
Namely, we obtain an equivalence of categories as below.
We denote by $\BEC_{\CC\mbox{\scriptsize -}c}(\I\CC_X)$ the triangulated category of $\CC$-constructible enhanced ind-sheaves on $X$.

\begin{theorem}[{\cite[Thm.\:3.26]{Ito20}}]\label{irregRH_ana}
There exists an equivalence of categories
\[\Sol_X^{\rmE} \colon \BDChol(\D_X)^{\op}\simto \BEC_{\CC\mbox{\scriptsize -}c}(\I\CC_X).\]
\end{theorem}
Moreover, the author proved an algebraic version of Theorem \ref{irregRH_ana} in \cite{Ito21}.
Let $X$ be a smooth algebraic variety
and denote by $\tl{X}$ a smooth completion of $X$.
The author defined algebraic $\CC$-constructibility for enhanced ind-sheaves
on a bordered space $X_\infty^\an = (X^\an, \tl{X}^\an)$
and proved the following result.
We denote by $\BEC_{\CC\mbox{\scriptsize -}c}(\I\CC_{X_\infty})$
the triangulated category of algebraic $\CC$-constructible enhanced ind-sheaves on $X_\infty^\an$.

\begin{theorem}[{\cite[Thm.\:3.11]{Ito21}}]\label{irregRH_alg}
There exists an equivalence of categories
\[\Sol_{X_\infty}^{\rmE} \colon \BDChol(\D_X)^{\op}\simto \BEC_{\CC\mbox{\scriptsize -}c}(\I\CC_{X_\infty}).\]
\end{theorem}

In this paper, we reprove Fact \ref{regRH_ana} (resp.\ Fact \ref{regRH_alg})
by using Fact \ref{irregRH_DK} and Theorem \ref{irregRH_ana} (resp.\ Theorem \ref{irregRH_alg})
in Theorem \ref{main2} (resp.\ Theorem \ref{main4}).
For this purpose, we study $\CC$-constructible enhanced ind-sheaves of
\cite[Def.\:3.19]{Ito20} (resp.\ \cite[Def.\:3.10]{Ito21})
in Propositions \ref{prop3.1}, \ref{prop3.2} (resp.\ Propositions \ref{prop3.7}, \ref{prop3.8}).
The key result of this paper is Lemma \ref{keylem}.

Note that the proofs of Theorems \ref{main1}, \ref{main2}, \ref{main3} and \ref{main4} are NOT circular reasoning.
The idea of the proof is in line with the one used by Z.\:Mebkhout
in the Riemann--Hilbert correspondence for regular holonomic $\D$-modules of \cite[Thm.\:2.1.1]{Meb84}.
Namely,
we reduce the problem to the case of regular meromorphic connections
by the d\'{e}vissage and the resolution singularity of \cite{Hiro64}.

\newpage
\section*{Acknowledgement}
I would like to thank Dr.\ Tauchi of Kyushu University for many discussions and giving many comments.

This work was supported by Grant-in-Aid for Research Activity Start-up (No.\ 21K20335)
and Grant-in-Aid for Young Scientists (No.\ 22K13902), 
Japan Society for the Promotion of Science.

\section{Preliminary Notions and Results}\label{sec-2}
%

\subsection{Bordered Spaces}
We shall recall a notion of bordered spaces.
See \cite[\S 3.2]{DK16} and \cite[2.1]{DK21} for the details.

A bordered space is a pair $M_{\infty} = (M, \che{M})$ of
a good topological space $\che{M}$ 
(i.e., a locally compact Hausdorff space which is countable at infinity and has finite soft dimension)
and an open subset $M\subset\che{M}$.
A morphism $f \colon (M, \che{M})\to (N, \che{N})$ of bordered spaces
is a continuous map $f \colon M\to N$ such that the first projection
$\che{M}\times\che{N}\to\che{M}$ is proper on
the closure $\var{\Gamma}_f$ of the graph $\Gamma_f$ of $f$ 
in $\che{M}\times\che{N}$. 
The category of good topological spaces is embedded into that
of bordered spaces by the identification $M = (M, M)$. 
Note that we have the morphism $j_{M_\infty} \colon M_\infty\to \che{M}$
of bordered spaces given by the embedding $M\hookrightarrow \che{M}$.
We sometimes denote $j_{M_\infty}$ by $j$ for short.
For a locally closed subset $Z\subset M$ of $M$,
we set $Z_\infty := (Z, \var{Z})$ where $\var{Z}$ is the closure of $Z$ in $\che{M}$
and denote by $i_{Z_\infty} \colon Z_\infty\to \var{Z}$ the morphism of bordered spaces
given by the embedding $Z\hookrightarrow \var{Z}$.

By definition, a subset of $M_\infty = (M, \che{M})$ is a subset of $M$.
We say that a subset $Z$ of $M_\infty$ is open (resp.\ closed, locally closed)
if it is so in $M$.
Moreover, a subset $Z$ of $M_\infty$ is a relatively compact subset
if it is contained in a compact subset of $ \che{M}$.

\subsection{Ind-Sheaves on Bordered Spaces}
We shall recall a notion of ind-sheaves on a bordered space of \cite[\S 3.2]{DK16}.

Let us denote by $\I\CC_{M_\infty}$ the abelian category of 
 ind-sheaves on a bordered space $M_{\infty} = (M, \che{M})$
 and denote by $\BDC(\I\CC_{M_\infty})$ the triangulated category of them.
For a morphism $f \colon M_\infty\to N_\infty$ 
of bordered spaces, 
we have the Grothendieck operations 
$ \otimes,\ \rihom,\ \rmR f_\ast,\ \rmR f_{!!},\ f^{-1},\ f^! $
for ind-sheaves on bordered spaces.
Note that there exists an embedding functor 
$\iota_{M_\infty} \colon \BDC(\CC_M) \hookrightarrow \BDC(\I\CC_{M_\infty})$.
We sometimes write $\BDC(\CC_{M_\infty})$ for $\BDC(\CC_{M})$,
when considered as the full subcategory of $\BDC(\I\CC_{M_\infty})$.
Note that there exists the standard t-structure
on $\BDC(\I\CC_{M_\infty})$.
Note also that the embedding functor $\iota_{M_\infty}$ has a left adjoint functor
$\alpha_{M_\infty} \colon \BDC(\I\CC_{M_\infty})\to \BDC(\CC_M)$.

\subsection{Enhanced Ind-Sheaves}\label{subsec2.3}
We shall recall some basic notions of enhanced ind-sheaves on bordered spaces and results on those.
Reference are made to \cite{KS16-2} and \cite{DK19, DK21}.
Moreover we also refer to \cite{DK16} and \cite{KS16}
for the notions of enhanced ind-sheaves on good topological spaces.

Let $M_\infty = (M, \che{M})$ be a bordered space.
We set $\RR_\infty := (\RR, \var{\RR})$ for 
$\var{\RR} := \RR\sqcup\{-\infty, +\infty\}$,
and let $t\in\RR$ be the affine coordinate. 
We consider the morphism of bordered spaces
$\pi\colon M_\infty \times\RR_\infty\to M_\infty$
given by the projection map $\pi\colon M\times \mathbb{R}\to M, (x,t)\mapsto x$. 
Then the triangulated category of enhanced ind-sheaves on a bordered space $M_\infty$ is defined by 
$$\BEC(\I\CC_{M_\infty}) :=
\BDC(\I\CC_{M_\infty \times\RR_\infty})/\pi^{-1}\BDC(\I\CC_{M_\infty}).$$
The quotient functor
$\Q_{M_\infty} \colon \BDC(\I\CC_{M_\infty\times\RR_\infty})\to\BEC(\I\CC_{M_\infty})$
has fully faithful left and right adjoints
$\bfL_{M_\infty}^\rmE,\bfR_{M_\infty}^\rmE \colon
\BEC(\I\CC_{M_\infty}) \to\BDC(\I\CC_{M_\infty\times\RR_\infty})$, respectively.
We sometimes denote $\Q_{M_\infty}$ (resp.\ $\bfL_{M_\infty}^\rmE,\ \bfR_{M_\infty}^\rmE$ )
by $\Q$ (resp.\ $\bfL^\rmE,\ \bfR^\rmE$) for short.
Then we have the standard t-structure 
on $\BEC(\I\CC_{M_\infty})$
which is induced by the standard t-structure on $\BDC(\I\CC_{M_\infty\times\RR_\infty})$.
We denote by $\bfE^0(\I\CC_{M_\infty})$
the heart of $\BEC(\I\CC_{M_\infty})$ with respect to the standard t-structure
and by $\SH^n \colon \BEC(\I\CC_{M_\infty})\to\bfE^0(\I\CC_{M_\infty})$
the $n$-th cohomology functor.
For a morphism $f \colon M_\infty\to N_\infty$ of bordered spaces, 
we have the Grothendieck operations 
$\Potimes,\ \Prihom, \bfE f^{-1},\ \bfE f_\ast,\ \bfE f^!,\ \bfE f_{!!}$
for enhanced ind-sheaves on bordered spaces.
Moreover,
for $F\in\BDC(\I\CC_{M_\infty})$ and $K\in\BEC(\I\CC_{M_\infty})$ the objects 
\begin{align*}
\pi^{-1}F\otimes K & :=\Q_{M_\infty}(\pi^{-1}F\otimes \bfL_{M_\infty}^\rmE K),\\
\rihom(\pi^{-1}F, K) & :=\Q_{M_\infty}\big(\rihom(\pi^{-1}F, \bfR_{M_\infty}^\rmE K)\big)
\end{align*}
in $\BEC(\I\CC_{M_\infty})$ are well defined. 
We set $$\CC_{M_\infty}^\rmE := \Q_{M_\infty} 
\Bigl(``\underset{a\to +\infty}{\varinjlim}"\ \CC_{\{t\geq a\}}
\Bigr)\in\BEC(\I\CC_{M_\infty})$$
where $\{t\geq a\}$ stands for $\{(x, t)\in M\times\RR\ |\ t\geq a\}\subset \che{M}\times\var{\RR}$.
Moreover, for a continuous function $\varphi \colon U\to \RR$ defined on an open subset $U\subset M$,
we set 
\[\EE_{U|M_\infty}^\varphi := 
\CC_{M_\infty}^\rmE\Potimes
\Q_{M_\infty}\big(\CC_{\{t+\varphi = 0\}}\big)
, \]
where $\{t+\varphi = 0\}$ stands for 
$\{(x, t)\in \che{M}\times\var{\RR}\ |\ t\in\RR, x\in U, 
t+\varphi(x) = 0\}$. 

We have a natural embedding
$e_{M_\infty} \colon \BDC(\I\CC_{M_\infty}) \to \BEC(\I\CC_{M_\infty})$
defined by \[e_{M_\infty}(F) :=  \CC_{M_\infty}^\rmE\otimes\pi^{-1}F,\]
see \cite[Lem.\:2.8.2]{DK19} (see also \cite[Prop.\:2.20]{KS16-2}) for the details.
Note also that for a morphism $f\colon M_\infty\to N_\infty$ of bordered spaces
and objects $F\in\BDC(\I\CC_{M_\infty})$, $G\in\BDC(\I\CC_{N_\infty})$
we obtain 
\begin{align*}
\bfE f_{!!}(e_{M_\infty}F)
\simeq
e_{N_\infty}(\bfR f_{!!}F),\hspace{19pt}
\bfE f^{-1}(e_{N_\infty}G)
\simeq
e_{M_\infty}(f^{-1}G),\hspace{19pt}
\bfE f^{!}(e_{N_\infty}G)
\simeq
e_{M_\infty}(f^{!}G)
\end{align*}
by using \cite[Prop.\:2.18]{KS16-2}.
Let $i_0 \colon M_\infty\to M_\infty\times\RR_\infty$ be the inclusion map of bordered spaces
induced by $M\to M\times \RR, x\mapsto (x, 0)$.
We set
\[\sh_{M_\infty} := \alpha_{M_\infty}\circ i_0^!\circ \bfR_{M_\infty}^{\rmE} \colon
\BEC(\I\CC_{M_\infty}) \to \BDC(\CC_{M}) \]
and call it the sheafification functor for enhanced ind-sheaves on bordered spaces.
We will use the following fact in \S 3.
\begin{fact}[{\cite[Prop.\:3.8 (i)]{DK21}}]\label{fact_she}
For any $\SF\in\BDC(\CC_M)$,
there exists an isomorphism
$$\SF\simto \sh_{M_\infty}(e_{M_\infty}(\iota_{M_\infty}(\SF))).$$
\end{fact}

The following notion was introduced in \cite{DK21}.
\begin{definition}[{\cite[Def.\:3.4 (ii)]{DK21}}]\label{def-sheaftype}
We say that $K\in \BEC(\I\CC_{M_\infty})$ is of sheaf type
if there exists an object $\F\in \BDC(\CC_{M_\infty})$
such that $K\simeq e_{M_\infty}(\iota_{M_\infty}(\F)))$.
\end{definition}

\subsection{$\RR$-Constructible Enhanced Ind-Sheaves} 
We shall recall a notion of the $\RR$-constructibility for enhanced ind-sheaves and results on those.
References are made to \cite{DK16, DK19}.

In this subsection,
we assume that a bordered space $M_\infty = (M, \che{M})$ is a subanalytic bordered space.
Namely, $\che{M}$ is a subanalytic space and $M$ is an open subanalytic subset of $\che{M}$.
See \cite[Def.\:3.1.1]{DK19} for the details.

\begin{definition}[{\cite[Def.\:3.1.2]{DK19}}]
We denote by $\BDC_{\RR\mbox{\scriptsize -}c}(\CC_{M_\infty})$
the full subcategory of $\BDC(\CC_{M_\infty})$
consisting of objects $\SF$ satisfying
$\rmR j_{M_\infty!}\SF$ is an $\RR$-constructible sheaf on $\che{M}$.
\end{definition}

Recall that a subset $Z$ of $M_\infty$ is subanalytic if it is subanalytic in $\che{M}$.
\begin{definition}[{\cite[Def.\:3.3.1]{DK19}}]\label{def2.3}
We say that $K\in\BEC(\I\CC_{M_\infty})$ is $\RR$-constructible
if for any relatively compact subanalytic open subset $U$ of $M_\infty$
there exists an isomorphism $\bfE i_{U_\infty}^{-1}K\simeq \CC_{U_\infty}^{\rmE}\Potimes \SF$
for some $\SF\in\BDC_{\RR\mbox{\scriptsize -}c}(\CC_{U_\infty\times\RR_\infty})$.
\end{definition}

We denote by $\BEC_{\RR\mbox{\scriptsize -}c}(\I\CC_{M_\infty})$
the full triangulated subcategory of $\BEC(\I\CC_{M_\infty})$
consisting of $\RR$-constructible enhanced ind-sheaves.
Note that the triangulated category $\BEC_{\RR\mbox{\scriptsize -}c}(\I\CC_{M_\infty})$
has the standard t-structure
which is induced by the standard t-structure on $\BEC(\I\CC_{M_\infty})$.
Let us denote by $\ZEC_{\RR\mbox{\scriptsize -}c}(\I\CC_{M_\infty})$
the heart of $\BEC_{\RR\mbox{\scriptsize -}c}(\I\CC_{M_\infty})$ with respect to the standard t-structure.

\subsection{$\D$-Modules}
In this section we recall some basic notions and results on $\D$-modules. 
References are made to 
\cite{Bjo93},
\cite[\S\S 8, 9]{DK16},
\cite[\S 7]{KS01},
\cite[\S\S 3, 4, 7]{KS16} for analytic $\D$-modules,
to \cite{Be}, \cite{Bor87}, \cite{HTT08} for algebraic ones.

\subsubsection{Analytic $\D$-Modules}\label{subsec2.4.1}
Let $X$ be a complex manifold
and denote by $d_X$ its complex dimension.
We denote by $\SO_{X}$ the sheaf of holomorphic functions and 
by $\D_{X}$ the sheaf of holomorphic differential operators on $X$.
Let us denote by $\BDC(\D_{X})$ the bounded derived category of left $\D_{X}$-modules. 
Moreover we denote by $\BDCcoh(\D_{X})$,
$\BDChol(\D_{X})$ and $\BDCrh(\D_{X})$ the full triangulated subcategories
of $\BDC(\D_{X})$ consisting of objects with coherent,
holonomic and regular holonomic cohomologies, respectively.
For a morphism $f \colon X\to Y$ of complex manifolds, 
denote by $\Dotimes,\ \bfD f_\ast,\ \bfD f^\ast$ 
the standard operations for analytic $\SD$-modules. 

For an analytic hypersurface $D$ in $X$ we denote by $\SO_{X}(\ast D)$ 
the sheaf of meromorphic functions on $X$ with poles in $D$. 
Then for $\M\in\BDC(\D_{X})$ we set 
$$\M(\ast D) := \M\Dotimes\SO_{X}(\ast D).$$
We say that a $\D_{X}$-module is a meromorphic connection on $X$ along $D$
if it is isomorphic as an $\SO_{X}$-module to a coherent $\SO_{X}(\ast D)$-module.
We denote by $\Conn({X}; D)$ the category of meromorphic connections along $D$
and by $\Conn^\reg({X}; D)$ the category of regular meromorphic connections along $D$.
Moreover, we set
\[\BDCmero(\D_{X(D)}) :=\{\M\in\BDChol(\D_{X})\
|\ \SH^i(\M)\in\Conn(X; D) \mbox{ for any }i\in\ZZ \}.\]

The classical solution functor on $X$ is defined by  
\begin{align*}
\Sol_{X} &\colon \BDCcoh (\D_{X})^{\op}\to\BDC(\CC_{X}),
\hspace{10pt}\M \longmapsto \rhom_{\D_{X}}(\M, \SO_{X}).
\end{align*}

An essential part of the following theorem was proved by Deligne in \cite{De70}.
We denote by ${\rm Loc}(X\setminus D)$ the category of local systems on $X\setminus D$.
The following theorem is used in the proof of Proposition \ref{prop3.1}.

\begin{fact}[{see e.g., \cite[Cor.\:5.2.21]{HTT08}}]\label{RH_Deligne}
There exists an equivalence of abelian categories
$$\mathcal{S}\colon \Conn^\reg({X}; D)\to {\rm Loc}(X\setminus D),\
\M\to \Sol_X(\M)|_{X\setminus D}.$$
\end{fact}

We denote by $\SO_{X}^{\rmE}$ the enhanced ind-sheaf of tempered holomorphic functions
\cite[Def.\:8.2.1]{DK16}
and by $\Sol_{X}^{\rmE}$ the enhanced solution functor on $X$:
\[
\Sol_{X}^\rmE \colon \BDCcoh (\D_{X})^{\op}\to\BEC(\I\CC_{X}), 
\hspace{10pt} 
\M \longmapsto \rihom_{\D_{X}}(\M, \SO_{X}^\rmE) ,
\]
\cite[Def.\:9.1.1]{DK16} (see also \cite[Lem.\:3.15]{Ito21}).
We will use the following facts in \S3.
\begin{fact}[{\cite[the equation just before Thm.\:9.1.2, Prop.\:9.1.3]{DK16}
(see also \cite[Last part of Prop.\:3.14]{Ito21})\footnote{
Remark that the assertion of \cite[Last part of Prop.\:3.14]{Ito21}
was proved without Fact \ref{regRH_alg}.}}]\label{fact_e}
For any $\M\in\BDCrh(\D_{X})$ there exists an isomorphism
$$\Sol_{X}^{\rmE}(\M)\simeq e_{X}\big(\Sol_{X}(\M)\big).$$
\end{fact}

\begin{fact}[{\cite[Lem.\:9.5.5]{DK16}}]\label{fact_sh}
For $\M\in\BDCcoh(\D_{X})$, 
we have an isomorphism
\[\sh_{X}\big( \Sol_{X}^{\rmE}(\M)\big)\simeq \Sol_{X}(\M).\]
\end{fact}

At the end of this subsection, let us recall the notion of $\M_{\reg}$.
We denote by $\D_{X}^\infty$
the sheaf of rings of differential operators of infinite order on $X$
and set $$\M^\infty := \D_{X}^\infty\otimes _{\D_{X}}\M.$$
Then for a holonomic $\D_{X}$-module $\M$,
a $\D_X$-module 
$$\M_{\reg} :=
\{s\in\M^\infty\ |\ \D_{X}\cdot s\in\Modrh(\D_{X})\}$$
is a regular holonomic $\D_{X}$-module.
Note that we have
$$(\M_{\reg})^\infty\simeq \M^\infty$$
and hence 
$$\Sol_{X}(\M_\reg)\simeq \Sol_{X}(\M).$$
See {\cite[Thm.\:5.2.1]{KK81}}, also {\cite[Prop.\:5.7]{Kas84}} for the details.

\subsubsection{Algebraic $\D$-Modules}
Let $X$ be a smooth algebraic variety over $\CC$
and denote by $d_X$ its complex dimension.
We denote by $\SO_X$ the sheaf of regular functions
and by $\SD_X$ the sheaf of algebraic differential operators on $X$.
Let us denote by $\BDC(\D_X)$ the bounded derived category of left $\D_X$-modules. 
Moreover we denote by $\BDCcoh(\D_X)$,
$\BDChol(\D_X)$ and $\BDCrh(\D_X)$ the full triangulated subcategories
of $\BDC(\D_X)$ consisting of objects with coherent,
holonomic and regular holonomic cohomologies, respectively.
For a morphism $f \colon X\to Y$ of smooth algebraic varieties, 
we denote by $\Dotimes,\ \bfD f_\ast,\ \bfD f^\ast$
the standard operations for algebraic $\SD$-modules. 

We denote by $X^\an$ the underlying complex manifold of $X$
and by $\tl{\iota} \colon (X^\an, \SO_{X^\an})\to(X, \SO_X)$ the morphism of ringed spaces.
Since there exists a morphism $\tl{\iota}^{-1}\SO_X\to\SO_{X^\an}$ of sheaves on $X^\an$,
we have a canonical morphism $\tl{\iota}^{-1}\D_X\to\D_{X^\an}$.
Then we set $$\M^\an := \D_{X^\an}\otimes_{\tl{\iota}^{-1}\D_X}\tl{\iota}^{-1}\M$$
for $\M\in\Mod(\D_X)$
and obtain a functor
$(\cdot)^\an \colon \Mod(\D_X)\to\Mod(\D_{X^\an}).$
It is called the analytification functor on $X$.
Since the sheaf $\D_{X^\an}$ is faithfully flat over $\tl{\iota}^{-1}\D_X$,
the analytification functor is faithful and exact,
and hence we obtain
$$(\cdot)^\an \colon \BDC(\D_X)\to\BDC(\D_{X^\an}).$$
Note that the analytification functor preserves the properties of coherent and holonomic.

At the end of this subsection, we shall recall algebraic meromorphic connections.
Let $D$ be a divisor of $X$,
and $j \colon X\setminus D\hookrightarrow X$ the natural embedding.
Then we set $\SO_X(\ast D) := j_\ast\SO_X$
and also set
$$\SM(\ast D) := \SM\Dotimes\SO_X(\ast D)$$
for $\SM\in\Mod(\SD_X)$.
Note that we have $\M(\ast D)\simeq \bfD j_\ast\bfD j^{\ast}\M$.
We say that a $\D_X$-module is a meromorphic connection on $X$ along $D$
if it is isomorphic as an $\SO_X$-module to a coherent $\SO_X(\ast D)$-module.
We denote by $\Conn(X; D)$ the category of meromorphic connections on $X$ along $D$.
Note that it is the full abelian subcategory of $\Modhol(\D_X)$.
Moreover, we set
\[\BDCmero(\D_{X(D)}) :=\{\M\in\BDChol(\D_X)\
|\ \SH^i(\M)\in\Conn(X; D) \mbox{ for any }i\in\ZZ \}.\]

We note that if $X$ is complete
there exists an equivalence of categories between
the abelian category $\Conn(X; D)$ and
the one of effective meromorphic connections on $X^\an$ along $D^\an$
by \cite[\S 5.3]{HTT08}.
However as a consequence of \cite[Thm.\:3.3.1]{Mal96} 
any analytic meromorphic connection is effective.
Hence we have:
\begin{fact}[{\cite[(5.3.2)]{HTT08}, \cite[Thm.\:3.3.1]{Mal96}}]\label{fact_mero}
If $X$ is complete, 
there exists an equivalence of abelian categories
$$(\cdot)^\an \colon \Conn(X; D)\simto\Conn(X^\an; D^\an).$$
Moreover this induces an equivalence of triangulated categories
$$(\cdot)^\an \colon \BDCmero(\D_{X(D)})\simto\BDCmero(\D_{X^\an(D^\an)}).$$
\end{fact}

\subsection{$\CC$-constructible Enhanced Ind-Sheaves}
In this section,  we recall the definition of $\CC$-constructibility for enhanced ind-sheaves
and main results of \cite{Ito20} and \cite{Ito21}.

\subsubsection{Analytic Case}
Let $X$ be a complex manifold and 
$D \subset X$ a normal crossing divisor on it. 
Let us take local coordinates 
$(u_1, \ldots, u_l, v_1, \ldots, v_{d_X-l})$ 
of $X$ such that $D= \{ u_1 u_2 \cdots u_l=0 \}$
and set $Y= \{ u_1=u_2= \cdots =u_l=0 \}.$
We define a partial order $\leq$ on the 
set $\ZZ^l$ by 
$$a \leq a^{\prime} \ \Longleftrightarrow 
\ a_i \leq a_i^{\prime} \ (1 \leq {}^\forall i \leq l),$$
for $a = (a_1, \ldots, a_l),\ a' = (a'_1, \ldots, a'_l) \in \ZZ^l$.
Then for a meromorphic function $\varphi\in\SO_{X}(\ast D)$
on $X$ along $D$ which has the Laurent expansion
\[ \varphi = \sum_{a \in \ZZ^l} c_a( \varphi )(v) \cdot 
u^a \ \in \SO_{X}(\ast D) \]
with respect to $u_1, \ldots, u_{l}$,
where $c_a( \varphi )$ are holomorphic functions on $Y$, 
we define its order 
$\ord( \varphi ) \in \ZZ^l$ by the minimum 
\[ \min \Big( \{ a \in \ZZ^l \ | \
c_a( \varphi ) \not= 0 \} \cup \{ 0 \} \Big) \]
if it exists. 
For any $f\in\SO_{X}(\ast D)/ \SO_{X}$, we take any lift $\tl{f}$ to $\SO_{X}(\ast D)$,
and we set $\ord(f) := \ord(\tl{f})$, if the right-hand side exists.
Note that it is independent of the choice of a lift $\tl{f}$.
If $\ord(f)\neq0$, $c_{\ord(f)}(\tl{f})$ is independent of the choice of a lift $\tl{f}$,
which is denoted by $c_{\ord(f)}(f)$.
\begin{definition}[{\cite[Def.\:2.1.2]{Mochi11}}]\label{def2.10}
In the situation as above,
a finite subset $\calI \subset \SO_{X}(\ast D)/ \SO_{X}$
is called a good set of irregular values on $(X,D)$,
if the following conditions are satisfied:
\begin{itemize}
\setlength{\itemsep}{-3pt}
\item[-]
For each element $f\in\calI$, $\ord(f)$ exists.
If $f\neq0$ in $\SO_{X}(\ast D)/ \SO_{X}$, $c_{\ord(f)}(f)$ is invertible on $Y$.
\item[-]
For two distinct $f, g\in\calI$, $\ord(f-g)$ exists and
$c_{\ord(f-g)}(f-g)$ is invertible on $Y$.
\item[-]
The set $\{\ord(f-g)\ |\ f, g\in\calI\}$ is totally ordered
with respect to the above partial order $\leq$ on $\ZZ^l$.
\end{itemize}
\end{definition}

\begin{definition}[{\cite[Def.\:3.5]{Ito20}}]\label{def-normal}
In the situation as above,
we say that an enhanced ind-sheaf
$K\in\ZEC(\I\CC_X)$ has a normal form along $D$
if the following three conditions are satisfied: 
\begin{itemize}
\setlength{\itemsep}{-1pt}
\item[(i)]
$\pi^{-1}\CC_{X\setminus D}\otimes K\simto K$,

\item[(ii)]
for any $x\in X\setminus D$ there exist an open neighborhood $U_x\subset X\setminus D$
of $x$ and a non-negative integer $k$ such that
$K|_{U_x}\simeq (\CC_{U_x}^{\rmE})^{\oplus k},$

\item[(iii)]
for any $x\in D$ there exist an open neighborhood $U_x\subset X$ of $x$,
a good set of irregular values $\{\varphi_i\}_i$ on $(U_x, D\cap U_x)$
and a finite sectorial open covering $\{U_{x, j}\}_j$ of $U_x\bs D$
such that
\[\pi^{-1}\CC_{U_{x, j}}\otimes K|_{U_x}\simeq
\bigoplus_i \EE_{U_{x, j} | U_x}^{\Re\varphi_i}
\hspace{10pt} \mbox{for any } j.\]
\end{itemize}
\end{definition}

In \cite[Def.\:3.5]{Ito20}, 
we assumed that $K$ is $\RR$-constructible,
see \cite[Def.\:3.3.1]{DK19} (see also Definition \ref{def2.3})
for the definition of $\RR$-constructible enhanced ind-sheaves.
However, it is not necessary:
\begin{proposition}
Any enhanced ind-sheaf which has a normal form along $D$
is an $\RR$-constructible enhanced ind-sheaf.
\end{proposition}

\begin{proof}
Let $K\in\ZEC(\I\CC_X)$ be an enhanced ind-sheaf which has a normal form along $D$.
Since the $\RR$-constructibility of enhanced ind-sheaves is a local property
(see \cite[Cor.\:4.9.8]{DK16} for details),
it is enough to show that for any $x\in X$ there exists an open subset $U_x\subset X$ of $x$
such that $K|_{U_x}
\in\ZEC_{\RR\mbox{\scriptsize -}c}(\I\CC_{U_x})$.

Since $K$ satisfies the condition (ii) in Definition \ref{def-normal} and
the constant enhanced ind-sheaf $\CC^\rmE$ is $\RR$-constructible,
for any $x\in X\setminus D$ there exists an open neighborhood $U_x\subset X\setminus D$ of $x$
such that $K|_{U_x}
\in\ZEC_{\RR\mbox{\scriptsize -}c}(\I\CC_{U_x})$.

Since $K$ satisfies the condition (iii) in Definition \ref{def-normal},
for any $x\in D$ there exist an open neighborhood $U_x\subset X$,
$L_x
\in\ZEC_{\RR\mbox{\scriptsize -}c}(\I\CC_{U_x})$ and 
a finite sectorial open covering $\{U_{x, j}\}$ of $U_x\setminus D$ such that
$$\pi^{-1}\CC_{U_{x, j}}\otimes K|_{U_x}\simeq \pi^{-1}\CC_{U_{x, j}}\otimes L_x$$
for any $j$.
Here we used the fact that the enhanced ind-sheaf
$\EE_{U_{x}\setminus D | U_x}^{\Re\varphi}$ is $\RR$-constructible
for any meromorphic function $\varphi\in\SO_{U_x}(\ast(D\cap U_x))$,
by Fact \ref{irregRH_DK} and \cite[Cor.\:9.4.12]{DK16}.
We shall show that $K|_{U_x}$ is $\RR$-constructible.
Note that since $K$ satisfies the condition (i) in Definition \ref{def-normal}
we have $$K|_{U_x}\simeq\pi^{-1}\CC_{U_x\setminus D}\otimes K|_{U_x}.$$
Hence by using \cite[Prop.\:4.9.3]{DK16} and
the Mayer--Vietoris sequence for sheaves (see e.g., \cite[Prop.\:2.3.6 (vii)]{KS90}),
it is enough to prove that
$\pi^{-1}\CC_{U_{x, j}}\otimes K|_{U_x}
\in\ZEC_{\RR\mbox{\scriptsize -}c}(\I\CC_{U_x})$.
However, this follows from $\pi^{-1}\CC_{U_{x, j}}\otimes L_x
\in\ZEC_{\RR\mbox{\scriptsize -}c}(\I\CC_{U_x})$.
\end{proof}

A ramification of $X$ along a normal crossing divisor $D$ on a neighborhood $U$ 
of $x \in D$ is a finite map $r \colon U^{\rami}\to U$ of complex manifolds of the form
$z' \mapsto z=(z_1,z_2, \ldots, z_n)= 
 r(z') = (z'^{m_1}_1,\ldots, z'^{m_k}_k, z'_{k+1},\ldots,z'_n)$ 
for some $(m_1, \ldots, m_k)\in (\ZZ_{>0})^k$, where 
$(z'_1,\ldots, z'_n)$ is a local coordinate system of $U^{\rami}$ and 
$(z_1, \ldots, z_n)$ is the one of 
$U$ such that $D \cap U=\{z_1\cdots z_k=0\}$. 

\begin{definition}[{\cite[Def.\:3.11]{Ito20}}]\label{def-quasi}
We say that an enhanced ind-sheaf
$K\in\ZEC(\I\CC_{X})$ has a quasi-normal form along $D$
if it satisfies (i) and (ii) in Definition \ref{def-normal}, and if
for any $x\in D$ there exist an open neighborhood $U_x\subset X$ of $x$
and a ramification $r_x \colon U_x^{\rami}\to U_x$ of $U_x$ along $D_x := U_x\cap D$
such that $\bfE r_x^{-1}(K|_{U_x})$ has a normal form along $D_x^{\rami}:= r_x^{-1}(D_x)$.
\end{definition}
Note that any enhanced ind-sheaf which has a quasi-normal form along $D$
is an $\RR$-constructible enhanced ind-sheaf on $X$.
See \cite[Prop.\:3.12]{Ito20} for the details. 

\newpage

A modification of $X$ with respect to an analytic hypersurface $H$
is a projective map $m \colon X^{\modi}\to X$ from a complex manifold $X^{\modi}$ to $X$ such that
$D^{\modi} := m^{-1}(H)$ is a normal crossing divisor of $X^{\modi}$
and $m$ induces an isomorphism $X^{\modi}\setminus D^{\modi}\simto X\setminus H$.

\begin{definition}[{\cite[Def.\:3.14]{Ito20}}]\label{def-modi}
We say that an enhanced ind-sheaf $K\in\ZEC(\I\CC_{X})$
has a modified quasi-normal form along $H$
if it satisfies (i) and (ii) in Definition \ref{def-normal}, and if
for any $x\in H$ there exist an open neighborhood $U_x\subset X$ of $x$
and a modification $m_x \colon U_x^{\modi}\to U_x$ of $U_x$ along $H_x := U_x\cap H$
such that $\bfE m_x^{-1}(K|_{U_x})$ has a quasi-normal form along $D_x^{\modi} := m_x^{-1}(H_x)$.
\end{definition}
Note that any enhanced ind-sheaf which has a modified quasi-normal form along $H$
is an $\RR$-constructible enhanced ind-sheaf on $X$. 
See \cite[Prop.\:3.15]{Ito20} for the details. 
Moreover we have:
\begin{proposition}[{\cite[Lem.\:3.16]{Ito20}}]\label{lem-modi}
The enhanced solution functor $\Sol_{X}^{\rmE}$ induces
an equivalence of abelian categories between
the full subcategory of $\ZEC(\I\CC_{X})$ consisting of
objects which have a modified quasi-normal form along $H$
and the abelian category $\Conn(X; H)$ of
meromorphic connections on $X$ along $H$.
\end{proposition}

We denote by $\ZECmero(\I\CC_{X(H)})$ the essential image of
$$\Sol_{X}^{\rmE} \colon \Conn(X; H)^{\op}\to\ZEC(\I\CC_{X}).$$
This abelian category is nothing but the full subcategory of $\ZEC(\I\CC_{X})$
consisting of enhanced ind-sheaves which have a modified quasi-normal form along $H$
by Proposition \ref{lem-modi}.
Moreover, we set
\begin{align*}
\BECmero(\I\CC_{X(H)}) &:=\{K\in\BEC_{\RR\mbox{\scriptsize -}c}(\I\CC_{X})\
|\ \SH^i(K)\in\ZECmero(\I\CC_{X(H)}) \mbox{ for any }i\in\ZZ \}.
\end{align*}

Since the category $\BDCmero(\D_{X(H)})$ is the full triangulated subcategory
of $\BDChol(\D_{X})$
and the category $\BECmero(\I\CC_{X(H)})$ is the full triangulated subcategory
of $\BEC_{\RR\mbox{\scriptsize -}c}(\I\CC_{X})$, the following proposition is obvious
by induction on the length of a complex:
\begin{proposition}\label{prop2.23}
The enhanced solution functor $\Sol_{X}^\rmE$ induces an equivalence of triangulated categories
$$\BDCmero(\D_{X(H)})^{\op} \simto \BECmero(\I\CC_{X(H)}).$$
\end{proposition}

A complex analytic stratification of $X$ is
a locally finite partition $\{X_\alpha\}_{\alpha\in A}$ of $X$
by locally closed analytic subsets $X_\alpha$
such that for any $\alpha\in A$, $X_\alpha$ is smooth,
$\var{X}_\alpha$ and $\partial{X_\alpha} := \var{X}_\alpha\setminus X_\alpha$
are complex analytic subsets 
and $\var{X}_\alpha = \bigsqcup_{\beta\in B} X_{\beta}$ for a subset $B\subset A$.
\begin{definition}[{\cite[Def.\:3.19]{Ito20}}]\label{def-const}
We say that an enhanced ind-sheaf $K\in\ZEC(\I\CC_{X})$ is $\CC$-constructible if
there exists a complex analytic stratification $\{X_\alpha\}_\alpha$ of $X$
such that
$$\pi^{-1}\CC_{\var{X}^{\blow}_\alpha\setminus D_\alpha}\otimes \bfE b_\alpha^{-1}K$$
has a modified quasi-normal form along $D_\alpha$ for any $\alpha$,
where $b_\alpha \colon \var{X}^{\blow}_\alpha \to X$ is a complex blow-up of $\var{X_\alpha}$
along $\partial X_\alpha = \var{X_\alpha}\setminus X_\alpha$ and
$D_\alpha := b_\alpha^{-1}(\partial X_\alpha)$.
Namely $\var{X}^{\blow} _\alpha$ is a complex manifold,
$D_\alpha$ is a normal crossing divisor of $\var{X}^{\blow} _\alpha$
and $b_\alpha$ is a projective map
which induces an isomorphism $\var{X}^{\blow} _\alpha\setminus D_\alpha\simto X_\alpha$
and satisfies $b_\alpha\big(\var{X}^{\blow} _\alpha\big)=\var{X_\alpha}$.

We call such a family $\{X_\alpha\}_{\alpha\in A}$
a complex analytic stratification adapted to $K$.
\end{definition}

\begin{remark}\label{rem2.25}
Definition \ref{def-const} does not depend on the choice of a complex blow-up $b_\alpha$
by \cite[Sublem.\:3.22]{Ito20}.
\end{remark}

We denote by $\ZEC_{\CC\mbox{\scriptsize -}c}(\I\CC_{X})$ the full subcategory of $\ZEC(\I\CC_{X})$
whose objects are $\CC$-constructible
and set
\[\BEC_{\CC\mbox{\scriptsize -}c}(\I\CC_{X}) := \{K\in\BEC(\I\CC_{X})\
|\ \SH^i(K)\in\ZEC_{\CC\mbox{\scriptsize -}c}(\I\CC_{X}) \mbox{ for any }i\in\ZZ \}\subset \BEC(\I\CC_{X}).\]
Note that 
the category $\BEC_{\CC\mbox{\scriptsize -}c}(\I\CC_{X})$ is
the full triangulated subcategory of $\BEC_{\RR\mbox{\scriptsize -}c}(\I\CC_{X})$.
See \cite[Prop.\:3.21]{Ito20} for the details.

\begin{theorem}[{\cite[Prop.\:3.25, Thm.\:3.26]{Ito20}}]\label{thm-Ito20}
For any $\M\in\BDChol(\D_{X})$, the enhanced solution complex $\Sol_{X}^\rmE(\M)$
of $\M$ is a $\CC$-constructible enhanced ind-sheaf.

On the other hand,
for any $\CC$-constructible enhanced ind-sheaf $K\in\BEC_{\CC\mbox{\scriptsize -}c}(\I\CC_{X})$,
there exists $\M\in\BDChol(\D_{X})$
such that $$K\simto \Sol_{X}^{\rmE}(\M).$$
Therefore we obtain an equivalence of triangulated categories
$$\Sol_{X}^{\rmE}\colon \BDChol(\D_{X})^{\op}\to \BEC_{\CC\mbox{\scriptsize -}c}(\I\CC_{X}).$$
\end{theorem}

\subsubsection{Algebraic Case}
Let $X$ be a smooth algebraic variety over $\CC$ and
denote by $X^\an$ the underlying complex analytic manifold of $X$.
Recall that an algebraic stratification of $X$ is
a Zariski locally finite partition $\{X_\alpha\}_{\alpha\in A}$ of $X$
by locally closed subvarieties $X_\alpha$
such that for any $\alpha\in A$, $X_\alpha$ is smooth and
$\var{X}_\alpha = \bigsqcup_{\beta\in B} X_{\beta}$ for a subset $B\subset A$.
Moreover an algebraic stratification $\{X_\alpha\}_{\alpha\in A}$ of $X$ induces
a complex analytic stratification $\{X^\an_\alpha\}_{\alpha\in A}$ of $X^\an$.

\begin{definition}[{\cite[Thm.\:3.1]{Ito21}}]
We say that an enhanced ind-sheaf $K\in\ZEC(\I\CC_{X^\an})$ satisfies the condition 
$\AC$
if there exists an algebraic stratification $\{X_\alpha\}_\alpha$ of $X$ such that
$$\pi^{-1}\CC_{(\var{X}^{\blow}_\alpha)^\an \setminus D^\an_\alpha}
\otimes \bfE (b^\an_\alpha)^{-1}K$$
has a modified quasi-normal form along $D^\an_\alpha$ for any $\alpha$,
where $b_\alpha \colon \var{X}^\blow_\alpha \to X$ is a blow-up of $\var{X_\alpha}$
along $\partial X_\alpha := \var{X_\alpha}\setminus X_\alpha$,
$D_\alpha := b_\alpha^{-1}(\partial X_\alpha)$
and $D_\alpha^\an := \big(\var{X}_\alpha^\blow\big)^\an\setminus
\big(\var{X}_\alpha^\blow\setminus D_\alpha\big)^\an$.
Namely $\var{X}^\blow_\alpha$ is a smooth algebraic variety over $\CC$,
$D_\alpha$ is a normal crossing divisor of $\var{X}^\blow_\alpha$
and $b_\alpha$ is a projective map
which induces an isomorphism $\var{X}^\blow_\alpha\setminus D_\alpha\simto X_\alpha$
and satisfies $b_\alpha\big(\var{X}^\blow_\alpha\big)=\var{X_\alpha}$.
\end{definition}

We denote by $\ZEC_{\CC\mbox{\scriptsize -}c}(\I\CC_X)$ the full subcategory of $\ZEC(\I\CC_{X^\an})$
whose objects satisfy the condition 
$\AC$.
Note that $\ZEC_{\CC\mbox{\scriptsize -}c}(\I\CC_X)$ is the full subcategory of
the abelian category $\ZEC_{\CC\mbox{\scriptsize -}c}(\I\CC_{X^\an})$
of $\CC$-constructible enhanced ind-sheaves on $X^\an$.
Moreover we set
\[\BEC_{\CC\mbox{\scriptsize -}c}(\I\CC_X) := \{K\in\BEC(\I\CC_{X^\an})\
|\ \SH^i(K)\in\ZEC_{\CC\mbox{\scriptsize -}c}(\I\CC_X) \mbox{ for any }i\in\ZZ \}
\subset \BEC_{\CC\mbox{\scriptsize -}c}(\I\CC_{X^\an}).\]

\begin{theorem}[{\cite[Thm.\:3.7]{Ito21}}]\label{thm-Ito21_complete}
Let $X$ be a smooth complete algebraic variety over $\CC$.
Then there exists an equivalence of triangulated categories
\[\Sol_X^{\rmE} \colon \BDChol(\D_X)^{\op}\simto \BEC_{\CC\mbox{\scriptsize -}c}(\I\CC_X),\
\M\mapsto \Sol_X^\rmE(\M) := \Sol_{X^\an}^\rmE(\M^\an).\]
\end{theorem}

Thanks to Hironaka's desingularization theorem \cite{Hiro64} (see also \cite[Thm\:4.3]{Naga62}),
we can take a smooth complete algebraic variety $\tl{X}$ such that $X\subset \tl{X}$
and $D := \tl{X}\setminus X$ is a normal crossing divisor of $\tl{X}$.
Let us consider a bordered space $$X^\an_\infty = (X^\an, \tl{X}^\an)$$
and the triangulated category $\BEC(\I\CC_{X^\an_\infty})$ of enhanced ind-sheaves
on $X^\an_\infty$.
Remark that $\BEC(\I\CC_{X^\an_\infty})$ does not depend on the choice of $\tl{X}$,
see \cite[\S 2.3]{Ito21} for the details.

We shall denote by $j\colon X\hookrightarrow \tl{X}$ the open embedding,
and by $j^\an\colon X^\an\hookrightarrow \tl{X}^\an$ the correspondence morphism for analytic spaces of $j$. 
Then we obtain the morphism of bordered spaces
$$j^\an\colon X^\an_\infty\to \tl{X}^\an$$
given by the embedding $j^\an\colon X^\an\hookrightarrow \tl{X}^\an$.

\begin{definition}[{\cite[Def.\:3.10]{Ito21}}]\label{def-algconst}
We say that an enhanced ind-sheaf $K\in\BEC(\I\CC_{X^\an_\infty})$ is
algebraic $\CC$-constructible on $X_\infty^\an$
if $\bfE j^\an_{!!}K \in\BEC(\I\CC_{\tl{X}^\an})$ is an object of $\BEC_{\CC\mbox{\scriptsize -}c}(\I\CC_{\tl{X}})$.
\end{definition}

We denote by $\BEC_{\CC\mbox{\scriptsize -}c}(\I\CC_{X_\infty})$
the full triangulated subcategory of $\BEC(\I\CC_{X^\an_\infty})$
consisting of algebraic $\CC$-constructible enhanced ind-sheaves on $X_\infty^\an$.
Note that the triangulated category $\BEC_{\CC\mbox{\scriptsize -}c}(\I\CC_{X_\infty})$ is the full triangulated subcategory
of $\BEC_{\RR\mbox{\scriptsize -}c}(\I\CC_{X^\an_\infty})$,
see \cite[Prop.\:3.21]{Ito20} for the details.

Let us set 
\[\Sol_{X_\infty}^\rmE(\M) := 
\bfE (j^\an)^{-1}\Sol_{\tl{X}}^\rmE(\bfD j_\ast\M)
\in\BEC(\I\CC_{X^\an_\infty})\]
for any $\M\in\BDC(\D_X)$.

\begin{theorem}[{\cite[Thm.\:3.11]{Ito21}}]\label{thm-Ito21}
For any $\M\in\BDChol(\D_X)$, the enhanced solution complex $\Sol_{X_\infty}^\rmE(\M)$
of $\M$ is an algebraic $\CC$-constructible enhanced ind-sheaf.

On the other hand, 
for any algebraic $\CC$-constructible enhanced ind-sheaf $K\in\BEC_{\CC\mbox{\scriptsize -}c}(\I\CC_{X_\infty})$,
there exists $\M\in\BDChol(\D_X)$
such that $$K\simto \Sol_{X_\infty}^{\rmE}(\M).$$
Moreover, we obtain an equivalence of triangulated categories
\[\Sol_{X_\infty}^{\rmE} \colon \BDChol(\D_X)^{\op}\simto \BEC_{\CC\mbox{\scriptsize -}c}(\I\CC_{X_\infty}).\]
\end{theorem}

\newpage
\section{Main Results}\label{sec3}
The main results of this paper are Theorems \ref{main1}, \ref{main2}, \ref{main3} and \ref{main4}.
\subsection{Analytic case}
In this subsection, let $X$ be a complex manifold.
First of all, we shall prove that the natural embedding functor $e_X\circ \iota_X$ and the sheafification functor $\sh_X$
preserve the $\CC$-constructibility.

Proposition \ref{prop3.1} (resp.\ Proposition \ref{prop3.2}) below was proved
in \cite[Cor.\:3.27]{Ito20} (resp.\ \cite[Cor.\:3.28]{Ito20})
by using Fact \ref{regRH_ana}.
In this paper, we will prove them without Fact \ref{regRH_ana}.

\begin{proposition}\label{prop3.1}
For any $\F\in\BDC_{\CC\mbox{\scriptsize -}c}(\CC_X)$,
we have $e_X(\iota_X(\SF))\in \BEC_{\CC\mbox{\scriptsize -}c}(\I\CC_X)$.
\end{proposition} 

\begin{proof}
By induction on the length of complex, it is enough to show
in the case when $\F$ is a $\CC$-constructible sheaf (not complex).

Let $\F$ be a $\CC$-constructible sheaf.
Then there exists a complex analytic stratification $\{X_\alpha\}_{\alpha\in A}$ of $X$
such that $\F|_{X_\alpha}$ is a local system. 
We shall prove that $K := e_X(\iota_X(\SF))$ is a $\CC$-constructible enhanced ind-sheaf.
For each $\alpha\in A$, let $b_\alpha\colon \var{X}_\alpha^\blow\to X$ be a complex blow-up of $\var{X}_\alpha$
along $\partial X_\alpha := \var{X}_\alpha\setminus X_\alpha$
and set $D_\alpha := b_\alpha^{-1}(\partial X_\alpha)$,
as in the condition (iii) of the definition of the $\CC$-constructibility (Definition \ref{def-const}).
Then we have isomorphisms
\begin{align*}
\pi^{-1}\CC_{\var{X}_\alpha^\blow\setminus D_\alpha}\otimes \bfE b_\alpha^{-1}K
&\simeq
e_{\var{X}_\alpha^\blow}\left(\iota_{\var{X}_\alpha^\blow}
\left(\CC_{\var{X}_\alpha^\blow\setminus D_\alpha}\otimes 
b_\alpha^{-1}(\F)\right)\right)\\
&\simeq
e_{\var{X}_\alpha^\blow}\left(\iota_{\var{X}_\alpha^\blow}
\left(i_{\var{X}_\alpha^\blow\setminus D_\alpha!}
(b_\alpha|_{\var{X}_\alpha^\blow\setminus D_\alpha})^{-1}(\F|_{X\alpha})\right)\right),
\end{align*}
by the commutativity of $e_X$ and $\iota_X$ for various operations, 
where $i_{\var{X}_\alpha^\blow\setminus D_\alpha}\colon \var{X}_\alpha^\blow\setminus D_\alpha\to \var{X}^\blow_\alpha$ is the natural embedding.
Since $(b_\alpha|_{\var{X}_\alpha^\blow\setminus D_\alpha})^{-1}(\F|_{X\alpha})$
is a local system on $\var{X}^\blow_\alpha\setminus D_\alpha$,
there exists an object $\M_\alpha\in\Conn^\reg(\var{X}^\blow_\alpha; D_\alpha)$
such that
$$(b_\alpha|_{\var{X}_\alpha^\blow\setminus D_\alpha})^{-1}(\F|_{X\alpha})\simeq
\Sol_{\var{X}^\blow_\alpha}(\M_\alpha)|_{\var{X}^\blow_\alpha\setminus D_\alpha}$$
by Fact \ref{RH_Deligne}.
Hence,
there exist isomorphisms
\begin{align*}
\pi^{-1}\CC_{\var{X}_\alpha^\blow\setminus D_\alpha}\otimes \bfE b_\alpha^{-1}K
&\simeq
e_{\var{X}_\alpha^\blow}\left(\iota_{\var{X}_\alpha^\blow}
\left(i_{\var{X}_\alpha^\blow\setminus D_\alpha!}\Sol_{\var{X}^\blow_\alpha}(\M_\alpha)
|_{\var{X}^\blow_\alpha\setminus D_\alpha}\right)\right)\\
&\simeq
e_{\var{X}_\alpha^\blow}\left(\iota_{\var{X}_\alpha^\blow}
\left(\CC_{\var{X}_\alpha^\blow\setminus D_\alpha}\otimes \Sol_{\var{X}^\blow_\alpha}(\M_\alpha)\right)\right)\\
&\simeq
e_{\var{X}_\alpha^\blow}\left(\iota_{\var{X}_\alpha^\blow}
\left(\Sol_{\var{X}^\blow_\alpha}(\M_\alpha)\right)\right)
\simeq
\Sol^\rmE_{\var{X}^\blow_\alpha}(\M_\alpha),
\end{align*}
where the third isomorphism follows from $\M_\alpha\simeq \M_\alpha(\ast D_\alpha)$
and the last isomorphism follows from Fact \ref{fact_e}.
Since $\M_\alpha\in \Conn(\var{X}^\blow_\alpha; D_\alpha)$,
the enhanced ind-sheaf $\Sol^\rmE_{\var{X}^\blow_\alpha}(\M)$
has a modified quasi-normal form along $D_\alpha$
by Proposition \ref{lem-modi}.

Therefore, the enhanced ind-sheaf
$$\pi^{-1}\CC_{\var{X}_\alpha^\blow\setminus D_\alpha}\otimes \bfE b_\alpha^{-1}K
\in\ZEC(\I\CC_{\var{X}_\alpha^\blow})$$
has a modified quasi-normal form along $D_\alpha$
for each $\alpha\in A$,
and hence the enhanced ind-sheaf $K = e_X(\iota_X(\SF))$ is $\CC$-constructible.
\end{proof}

\begin{proposition}\label{prop3.2}
For any $K\in \BEC_{\CC\mbox{\scriptsize -}c}(\I\CC_X)$,
we have $\sh_X(K)\in\BDC_{\CC\mbox{\scriptsize -}c}(\CC_X)$.
\end{proposition} 

\begin{proof}
By induction on the length of complex, it is enough to show in the case of $K\in\ZEC_{\CC\mbox{\scriptsize -}c}(\I\CC_X)$.

Let $K\in\ZEC_{\CC\mbox{\scriptsize -}c}(\I\CC_X)$
and $\{X_\alpha\}_{\alpha\in A}$ a complex analytic stratification adapted to $K$.
We shall prove that $\sh_X(K)|_{X_\alpha}$ is a local system for each $\alpha\in A$.
For each $\alpha\in A$, let $b_\alpha\colon \var{X}_\alpha^\blow\to X$ be a complex blow-up of $\var{X}_\alpha$
along $\partial X_\alpha := \var{X}_\alpha\setminus X_\alpha$
and set $D_\alpha := b_\alpha^{-1}(\partial X_\alpha)$,
as in the condition (iii) of the definition of the $\CC$-constructibility (Definition \ref{def-const}).
Since $b_\alpha|_{\var{X}_\alpha^\blow\setminus D_\alpha}\colon
\var{X}^{\blow} _\alpha\setminus D_\alpha\simto X_\alpha$
is an isomorphism,
it is enough to show that $(b_\alpha|_{\var{X}_\alpha^\blow\setminus D_\alpha})^{-1}
(\sh_X(K)|_{X_\alpha})$ is a local system on $\var{X}_\alpha^\blow\setminus D_\alpha$.
However, this follows from 
$$(b_\alpha|_{\var{X}_\alpha^\blow\setminus D_\alpha})^{-1}
(\sh_X(K)|_{X_\alpha})\simeq
\sh_{\var{X}_\alpha^\blow\setminus D_\alpha}
\left((\bfE b_\alpha^{-1}K)|_{\var{X}_\alpha^\blow\setminus D_\alpha}\right)$$
by \cite[Lem.\:3.3 (1)]{DK21},
the condition (ii) in Definition \ref{def-modi} (see also Definition \ref{def-normal})
and the fact that there exists an isomorphism
$\sh_{M_\infty}(\CC_{M_\infty}^\rmE) \simeq \CC_{M}$
for any bordered space $M_\infty = (M, \che{M})$.
\end{proof}

The following theorem can be proved as a corollary of \cite[Thm.\:4.8]{Kas78} (see also \cite[Thm.\:3.5]{Kas75}).
In this paper, we will give an another proof by using Proposition \ref{prop3.2}.
\begin{theorem}\label{main1}
For any $\M\in\BDChol(\D_X)$, we have $\Sol_X(\M)\in \BDC_{\CC\mbox{\scriptsize -}c}(\CC_X)$.
\end{theorem}

\begin{proof}
Let $\M\in\BDChol(\D_X)$.
Then we have $\Sol_X^\rmE(\M) \in \BEC_{\CC\mbox{\scriptsize -}c}(\I\CC_X)$ by Theorem \ref{thm-Ito20}.
By Fact \ref{fact_sh}, we have an isomorphism
$$\Sol_X(\M)\simeq \sh_X(\Sol_X^\rmE(\M))$$
and hence 
we obtain $\Sol_X(\M)\in \BDC_{\CC\mbox{\scriptsize -}c}(\CC_X)$
by Proposition \ref{prop3.2}.
\end{proof}

The following lemma is a key lemma of this paper.
\begin{lemma}\label{keylem}
Let $\M\in\BDChol(\D_X)$.
An enhanced ind-sheaf
$\Sol_X^\rmE(\M)$ is of sheaf type if and only if $\M\in\BDCrh(\D_X)$.
\end{lemma} 

\begin{proof}
By Fact \ref{fact_e}, 
an enhanced ind-sheaf $\Sol_X^\rmE(\M)$ is of sheaf type if $\M\in\BDCrh(\D_X)$.

We assume that 
$\Sol_X^\rmE(\M)$ is of sheaf type.
By definition (see Definition \ref{def-sheaftype}), there exists $\F\in \BDC(\CC_X)$ such that 
$$\Sol_X^\rmE(\M) \simeq e_X(\iota_X(\F)).$$
By Fact \ref{fact_she} and Fact \ref{fact_sh}, 
we have $\Sol_X(\M)\simeq \F$.
Remark that there exists an isomorphism $\Sol_X(\M)\simeq \Sol_X(\M_\reg)$,
see the end of \S \ref{subsec2.4.1}.
Hence, we have isomorphisms
$$\Sol_X^\rmE(\M) \simeq e_X(\iota_X(\F))
\simeq e_X(\iota_X(\Sol_X(\M_\reg)))
\simeq \Sol_X^\rmE(\M_\reg),$$
where the last isomorphism follows from Fact \ref{fact_e}.
Therefore we have $\M\simeq \M_\reg$ by Fact \ref{irregRH_DK}
and hence $\M\in \BDCrh(\D_X)$.
\end{proof}

Let us reprove Fact \ref{regRH_ana}
(the Riemann--Hilbert correspondence for regular holonomic $\D$-modules of \cite{Kas84}).
\begin{theorem}\label{main2}
There exists an equivalence of triangulated categories
$$\Sol_X\colon  \BDCrh(\D_X)^{\op}\simto\BDC_{\CC\mbox{\scriptsize -}c}(\CC_X).$$
\end{theorem}

\begin{proof}
By Theorem \ref{main1},
it is enough to show that the functor
$\Sol_X\colon  \BDCrh(\D_X)^{\op}\to\BDC_{\CC\mbox{\scriptsize -}c}(\CC_X)$
is fully faithful and essentially surjective.

Let $\M, \N\in \BDCrh(\D_X)$.
Then we have isomorphisms
\begin{align*}
\Hom_{\BDC_{\CC\mbox{\scriptsize -}c}(\CC_X)}\left(\Sol_X(\N), \Sol_X(\M)\right)
&\simeq
\Hom_{\BEC_{\CC\mbox{\scriptsize -}c}(\I\CC_X)}\left(e_X(\iota_X(\Sol_X(\N))), e_X(\iota_X(\Sol_X(\M)))\right)\\
&\simeq
\Hom_{\BEC_{\CC\mbox{\scriptsize -}c}(\I\CC_X)}\left(\Sol_X^\rmE(\N), \Sol_X^\rmE(\M)\right)\\
&\simeq
\Hom_{\BDCrh(\D_X)}\left(\M, \N\right),
\end{align*}
where the first isomorphism follows from
the fact that the functor
$e_X\circ \iota_X\colon \BDC_{\CC\mbox{\scriptsize -}c}(\CC_X)\to \BEC_{\CC\mbox{\scriptsize -}c}(\I\CC_X)$
is fully faithful by \cite[Prop.\:4.7.15]{DK16}, \cite[Prop.\:3.3.4]{KS01} and Proposition \ref{prop3.1},
the second isomorphism follows from Fact \ref{fact_e}
and the last isomorphism follows from Fact \ref{irregRH_DK}.
Hence, the functor $\Sol_X$ is fully faithful.

Let $\F\in \BDC_{\CC\mbox{\scriptsize -}c}(\CC_X)$.
By Proposition \ref{prop3.1}, we have $e_X(\iota_X(\F)) \in \BEC_{\CC\mbox{\scriptsize -}c}(\I\CC_X)$
and hence there exists $\M\in\BDChol(\D_X)$ such that $e_X(\iota_X(\F)) \simeq \Sol_X^\rmE(\M)$
by Theorem \ref{thm-Ito20}.
Since $\M\in \BDCrh(\D_X)$ by Lemma \ref{keylem},
we obtain an isomorphism
$$e_X(\iota_X(\F)) \simeq e_X(\iota_X(\Sol_X(\M)))$$
by Fact \ref{fact_e}
and hence we have $\F\simeq \Sol_X(\M)$
by applying the sheafification functor $\sh_X$ and using Fact \ref{fact_she}.
This means that the functor $\Sol_X$ is essentially surjective.

Therefore, there exists an equivalence of triangulated categories
$$\Sol_X\colon  \BDCrh(\D_X)^{\op}\simto\BDC_{\CC\mbox{\scriptsize -}c}(\CC_X).$$
\end{proof}

\subsection{Algebraic case}
In this subsection, let $X$ be a smooth algebraic variety over $\CC$.
First of this subsection, we shall prove that the natural embedding functor $e_{X^\an_\infty}\circ \iota_{X^\an_\infty}$
and the sheafification functor $\sh_{X^\an_\infty}$ preserve the algebraic $\CC$-constructibility.

Proposition \ref{prop3.7} (resp.\ Proposition \ref{prop3.8}) below was proved
in \cite[Prop.\:3.14]{Ito21} (resp.\ \cite[Prop.\:3.16]{Ito21})
by using Fact \ref{regRH_alg}.
In this paper, we will prove them without Fact \ref{regRH_alg}.

The following lemma can be prove by using Fact \ref{fact_mero} and 
the same arguments of Propositions \ref{prop3.1}, \ref{prop3.2}.
We shall skip the proof of this lemma.
\begin{lemma}\label{lem3.6}
If $X$ is complete then we have:
\begin{itemize}
\item[\rm (1)]
For any $\F\in\BDC_{\CC\mbox{\scriptsize -}c}(\CC_X)$,
we have $e_{X^\an}(\iota_{X^\an}(\SF))\in \BEC_{\CC\mbox{\scriptsize -}c}(\I\CC_X)$.

\item[\rm (2)]
For any $K\in \BEC_{\CC\mbox{\scriptsize -}c}(\I\CC_X)$,
we have $\sh_{X^\an}(K)\in\BDC_{\CC\mbox{\scriptsize -}c}(\CC_X)$.
\end{itemize}
\end{lemma}

\newpage
Again, let $X$ be a smooth algebraic variety (not necessarily complete) over $\CC$.
\begin{proposition}\label{prop3.7}
For any $\F\in\BDC_{\CC\mbox{\scriptsize -}c}(\CC_X)$,
we have $e_{X^\an_\infty}(\iota_{X^\an_\infty}(\SF))\in \BEC_{\CC\mbox{\scriptsize -}c}(\I\CC_{X_\infty})$.
\end{proposition} 

\begin{proof}
Let $\F\in\BDC_{\CC\mbox{\scriptsize -}c}(\CC_X)$ and
we set $K := e_{X^\an_\infty}(\iota_{X^\an_\infty}(\SF))\in \BEC(\I\CC_{X^\an_\infty})$.
It is enough to show that $\bfE j^\an_{!!}K\in \BEC_{\CC\mbox{\scriptsize -}c}(\I\CC_{\tl{X}})$
by Definition \ref{def-algconst}.
Since $j^\an\colon X^\an_\infty\to \tl{X}^\an$ is semi-proper
(see \cite[Def.\:2.3.5]{DK19} (also  \cite[Def.\:2.4]{KS16-2}) for the definition of semi-proper),
there exists an isomorphism
$$\bfE j^\an_{!!}K\simeq e_{\tl{X}^\an}(\iota_{\tl{X}^\an}(\bfR j^\an_!(\SF)))$$
by  \cite[Prop.\:2.18 (i)]{KS16-2} and \cite[Rem.\:2.4.3]{DK19}.
Since $\bfR j^\an_!(\SF)\in \BDC_{\CC\mbox{\scriptsize -}c}(\CC_{\tl{X}})$ by \cite[Thm.\:4.5.8 (iii)]{HTT08},
we have $$e_{\tl{X}^\an}(\iota_{\tl{X}^\an}(\bfR j^\an_!(\SF)))
\in \BEC_{\CC\mbox{\scriptsize -}c}(\I\CC_{\tl{X}})$$
by Lemma \ref{lem3.6} (1).
Therefore, we have $\bfE j^\an_{!!}K\in \BEC_{\CC\mbox{\scriptsize -}c}(\I\CC_{\tl{X}})$
and hence $K = e_{X^\an_\infty}(\iota_{X^\an_\infty}(\SF))\in \BEC_{\CC\mbox{\scriptsize -}c}(\I\CC_{X_\infty})$.
\end{proof}

\begin{proposition}\label{prop3.8}
For any $K\in \BEC_{\CC\mbox{\scriptsize -}c}(\I\CC_{X_\infty})$,
we have $\sh_{X^\an_\infty}(K)\in\BDC_{\CC\mbox{\scriptsize -}c}(\CC_X)$.
\end{proposition} 

\begin{proof}
Let $K\in \BEC_{\CC\mbox{\scriptsize -}c}(\I\CC_{X_\infty})$.
Recall that there exists an isomorphism
$$\sh_{X^\an_\infty}(K)
\simeq 
(j^{-1})^\an(\sh_{\tl{X}^\an}(\bfE j^\an_{!!}K))
$$
by the definition of the sheafification functor.
By the definition of the triangulated category $\BEC_{\CC\mbox{\scriptsize -}c}(\I\CC_{X_\infty})$
(see Definition \ref{def-algconst}),
we have $\bfE j^\an_{!!}K\in \BEC_{\CC\mbox{\scriptsize -}c}(\I\CC_{\tl{X}})$,
and hence $$\sh_{\tl{X}^\an}(\bfE j^\an_{!!}K)\in\BDC_{\CC\mbox{\scriptsize -}c}(\CC_{\tl{X}})$$
by Lemma \ref{lem3.6} (2).
Moreover by \cite[Thm.\:4.5.8 (ii)]{HTT08}, 
we have $$(j^{-1})^\an(\sh_{\tl{X}^\an}(\bfE j^\an_{!!}K))\in\BDC_{\CC\mbox{\scriptsize -}c}(\CC_{X}).$$
Therefore we have $\sh_{X^\an_\infty}(K) \in \BDC_{\CC\mbox{\scriptsize -}c}(\CC_{X})$.
\end{proof}

The following theorem was proved in \cite[Main Theorem C (a)]{Be}.
In this paper, we will give an another proof by using Proposition \ref{prop3.8}.
\begin{theorem}\label{main3}
For any $\M\in\BDChol(\D_X)$, we have $\Sol_X(\M)\in \BDC_{\CC\mbox{\scriptsize -}c}(\CC_X)$.
\end{theorem}

\begin{proof}
Let $\M\in\BDChol(\D_X)$.
Then we have $\Sol_{X_\infty}^\rmE(\M) \in \BEC_{\CC\mbox{\scriptsize -}c}(\I\CC_{X_\infty})$
by Theorem \ref{thm-Ito21}.
By \cite[Lem.\:3.15]{Ito21}, we have an isomorphism
$$\Sol_X(\M)\simeq \sh_{X^\an_\infty}(\Sol_{X_\infty}^\rmE(\M))$$
and hence we obtain $\Sol_X(\M)\in \BDC_{\CC\mbox{\scriptsize -}c}(\CC_X)$
by Proposition \ref{prop3.8}.
\end{proof}

\begin{lemma}\label{lem3.10}
Let $\M\in\BDChol(\D_X)$.
An enhanced ind-sheaf
$\Sol_{X_\infty}^\rmE(\M)$ is of sheaf type if and only if $\M\in\BDCrh(\D_X)$.
\end{lemma} 

\begin{proof}
By \cite[Last part of Prop.\:3.14]{Ito21}\footnote{
Remark that the assertion of \cite[Last part of Prop.\:3.14]{Ito21} (where we omit iota in the diagram)
was proved without Fact \ref{regRH_alg}.}, 
an enhanced ind-sheaf $\Sol_{X_\infty}^\rmE(\M)$ is of sheaf type if $\M\in\BDCrh(\D_X)$.

We assume that 
$\Sol_{X_\infty}^\rmE(\M)$ is of sheaf type.
By definition (see Definition \ref{def-sheaftype}),
there exists $\F\in \BDC(\CC_X)$ such that 
$$\Sol_{X_\infty}^\rmE(\M) \simeq e_{X^\an_\infty}(\iota_{X^\an_\infty}(\F)).$$
By applying the functor $\bfE j^\an_{!!}$,
we have an isomorphism 
$$\Sol_{\tl{X}^\an}^\rmE((\bfD j_\ast\M)^\an) \simeq e_{\tl{X}^\an}(\iota_{\tl{X}^\an}(\bfR j^\an_!(\F))).$$
Hence, 
the enhanced ind-sheaf
$\Sol_{\tl{X}^\an}^\rmE((\bfD j_\ast\M)^\an)\in\BEC(\I\CC_{\tl{X}^\an})$
is of sheaf type.
By Lemma \ref{keylem},
we have $(\bfD j_\ast\M)^\an\in\BDCrh(\D_{\tl{X}^\an})$.
This means that $\M\in \BDCrh(\D_X)$
by \cite[Thm.\:6.1.12]{HTT08}.
\end{proof}

Let us reprove Fact \ref{regRH_alg}
(the algebraic version of the Riemann--Hilbert correspondence for regular holonomic $\D$-modules).
\begin{theorem}\label{main4}
There exists an equivalence of triangulated categories
$$\Sol_X\colon  \BDCrh(\D_X)^{\op}\simto\BDC_{\CC\mbox{\scriptsize -}c}(\CC_X).$$
\end{theorem}

\begin{proof}
By Theorem \ref{main3}
it is enough to show that the functor
$\Sol_X\colon  \BDCrh(\D_X)^{\op}\to\BDC_{\CC\mbox{\scriptsize -}c}(\CC_X)$
is fully faithful and essentially surjective.

Let $\M, \N\in \BDCrh(\D_X)$.
Then we have isomorphisms
\begin{align*}
\Hom_{\BDC_{\CC\mbox{\scriptsize -}c}(\CC_X)}\left(\Sol_X(\N), \Sol_X(\M)\right)
&\simeq
\Hom_{\BEC_{\CC\mbox{\scriptsize -}c}(\I\CC_{X_\infty})}\left(
e_{X^\an_\infty}(\iota_{X^\an_\infty}(\Sol_X(\N))), e_{X^\an_\infty}(\iota_{X^\an_\infty}(\Sol_X(\M)))\right)\\
&\simeq
\Hom_{\BEC_{\CC\mbox{\scriptsize -}c}(\I\CC_{X_\infty})}\left(
\Sol_{X_\infty}^\rmE(\N), \Sol_{X_\infty}^\rmE(\M)\right)\\
&\simeq
\Hom_{\BDCrh(\D_X)}\left(\M, \N\right),
\end{align*}
where the first isomorphism follows from
the fact that the functor
$e_{X^\an_\infty}\circ \iota_{X^\an_\infty}\colon
\BDC_{\CC\mbox{\scriptsize -}c}(\CC_X)\to \BEC_{\CC\mbox{\scriptsize -}c}(\I\CC_{X_\infty})$
is fully faithful by \cite[Lem.\:2.8.2]{DK19} (see also \cite[Prop.\:2.20]{KS16-2}),
\cite[(2.6)]{KS16-2}
and Proposition \ref{prop3.7},
the second isomorphism follows from \cite[Last part of Prop.\:3.14]{Ito21}\footnote{
Remark that the assertion of \cite[Last part of Prop.\:3.14]{Ito21} (where we omit iota in the diagram)
was proved without Fact \ref{regRH_alg}.}
and the last isomorphism follows from Theorem \ref{thm-Ito21}.
Hence, the functor $\Sol_X$ is fully faithful.

Let $\F\in \BDC_{\CC\mbox{\scriptsize -}c}(\CC_X)$.
By Proposition \ref{prop3.7},
we have $e_{X^\an_\infty}(\iota_{X^\an_\infty}(\F)) \in \BEC_{\CC\mbox{\scriptsize -}c}(\I\CC_{X_\infty})$
and hence there exists $\M\in\BDChol(\D_X)$ such that
$e_{X^\an_\infty}(\iota_{X^\an_\infty}(\F)) \simeq \Sol_{X_\infty}^\rmE(\M)$
by Theorem \ref{thm-Ito21}.
Since $\M\in \BDCrh(\D_X)$ by Lemma \ref{lem3.10},
we obtain an isomorphism
$$e_{X^\an_\infty}(\iota_{X^\an_\infty}(\F)) \simeq e_{X^\an_\infty}(\iota_{X^\an_\infty}(\Sol_X(\M)))$$
by \cite[Last part of Prop.\:3.14]{Ito21}
and hence we have $\F\simeq \Sol_X(\M)$
by applying the sheafification functor $\sh_{X^\an_\infty}$
and using Fact \ref{fact_she}.
This means that the functor $\Sol_X$ is essentially surjective.

Therefore, there exists an equivalence of triangulated categories
$$\Sol_X\colon  \BDCrh(\D_X)^{\op}\simto\BDC_{\CC\mbox{\scriptsize -}c}(\CC_X).$$
\end{proof}


\begin{thebibliography}{9}
\bibitem[Be]{Be}
Joseph Bernstein, Algebraic Theory of $\D$-Modules, unpublished notes.

\bibitem[Bj{\"o}93]{Bjo93}
Jan-Erik Bj{\"o}rk,
Analytic $\D$-modules and applications,
 Mathematics and its Applications, {\bf247}, 
Kluwer Academic Publishers Group, 1993.


\bibitem[Bor87]{Bor87}
Armand Borel et al.,
Algebraic D-Modules,
Perspectives in Mathematics, {\bf2}, Academic Press, 1987.

\bibitem[De70]{De70}
P. Deligne,
\'{E}quations diff\'{e}rentielles \'{a} points singuliers r\'{e}guliers,
Lecture Notes in Mathematics, {\bf 163}, Springer-Verlag, 1970.


\bibitem[DK16]{DK16}
Andrea D'Agnolo and Masaki Kashiwara,
Riemann--Hilbert correspondence for holonomic $\D$-modules,
Publ. Math. Inst. Hautes \'{E}tudes Sci., {\bf 123(1)}, 69--197, 2016.


\bibitem[DK19]{DK19}
Andrea D'Agnolo and Masaki Kashiwara,
Enhanced perversities, J. Reine Angew. Math. (Crelle's Journal), {\bf751}, 185--241, 2019.

\bibitem[DK21]{DK21}
Andrea D'Agnolo and Masaki Kashiwara,
On a topological counterpart of regularization for holonomic $\D$-modules,
Journal de l'\'Ecole polytechnique Math\'ematiques, {\bf 8}, 27--55, 2021

\bibitem[Hiro64]{Hiro64}
Heisuke Hironaka,
Resolution of singularities of an algebraic variety over a field of characteristic zero I, II,
Ann. Math., {79(1)}, I:109--203, 1964, II: 205--326, 1964.

\bibitem[HTT08]{HTT08}
Ryoshi Hotta, Kiyoshi Takeuchi, and Toshiyuki Tanisaki,
$\D$-modules, perverse sheaves, and representation theory,
Progress in Mathematics, {\bf236}, Birkh{\"a}user, 2008.

\bibitem[Ito20]{Ito20}
Yohei Ito, $\CC$-Conctructible Enhanced Ind-Sheaves,
Tsukuba journal of Mathematics, {\bf44(1)}, 155--201, 2020.

\bibitem[Ito21a]{Ito21}
Yohei Ito,
Note on Algebraic Irregular Riemann--Hilbert Correspondence,
Rend. Sem. Mat. Univ. Padova., 2021, in press, arXiv:2004.13518, 52 pages.

\bibitem[Ito21b]{Ito21b}
Note on Relation between Enhanced Ind-Sheaves and Enhanced Subanalytic Sheaves,
arXiv:2109.13991, 84 pages.

\bibitem[Kas75]{Kas75}
Masaki Kashiwara,
On the maximally overdetermined system of linear differential equations, I,
Publ. RIMS, Kyoto Univ.,{\bf10}, 563--579, 1975.

\bibitem[Kas78]{Kas78}
Masaki Kashiwara,
On the holonomic systems of linear differential equations, II,
Inventiones math., {\bf49}, 121--135, 1978.

\bibitem[Kas84]{Kas84}
Masaki Kashiwara,
The Riemann--Hilbert problem for holonomic systems,
Publ. Res. Inst. Math. Sci., {\bf20(2)}, 319--365, 1984.

\bibitem[Kas16]{Kas16}
Masaki Kashiwara,
Riemann--Hilbert correspondence for irregular holonomic $\D$-modules,
Japan J. Math., {\bf 11}, 13--149, 2016.

\bibitem[KK81]{KK81}
M.\:Kashiwara and Takahiro\:Kawai,
On holonomic systems of microdifferential equations
III—system with regular singularities, Publ. Res. Inst. Math. Sci., {\bf 17}, 813--979, 1981.


\bibitem[KS90]{KS90}
Masaki Kashiwara and Pierre Schapira,
Sheaves on manifolds, Grundlehren der Mathematischen Wissenschaften,
{\bf292}, Springer-Verlag, 1990.

\bibitem[KS01]{KS01}
Masaki Kashiwara and Pierre Schapira,
Ind-sheaves,
Ast\'{e}risque, {\bf 271}, 2001.

\bibitem[KS06]{KS06}
Masaki Kashiwara and Pierre Schapira,
Categories and Sheaves,
Grundlehren der Mathematischen Wissenschaften, {\bf332} ,
Springer-Verlag, 2006.

\bibitem[KS16a]{KS16-2}
Masaki Kashiwara and Pierre Schapira,
 Irregular holonomic kernels and Laplace transform,
Selecta Math., {\bf22(1)}, 55--109, 2016.

\bibitem[KS16b]{KS16}
Masaki Kashiwara and Pierre Schapira,
Regular and irregular holonomic $\D$-modules, 
 London Mathematical Society Lecture Note Series, {\bf433},
 Cambridge University Press, 2016.
 
\bibitem[Ked10]{Ked10}
Kiran~S. Kedlaya,
Good formal structures for flat meromorphic connections, {I}:
  surfaces, Duke Math. J., {\bf154(2)}, 343--418, 2010.

\bibitem[Ked11]{Ked11}
Kiran~S. Kedlaya,
Good formal structures for flat meromorphic connections, {II}:
 excellent schemes, J. Amer. Math. Soc., {\bf24(1)}, 183--229, 2011.

 \bibitem[Kuwa18]{Kuwa18}
  Tatsuki Kuwagaki, 
  Irregular perverse sheaves, 
  Compositio Mathematica, {\bf 157(3)}, 573--624, 2021.
  
 \bibitem[Mal96]{Mal96}
 Bernard Malgrange,
 “Connexions m{\'e}romorphes, II: le r{\'e}seau canonique”,
 Invent. Math. 124, 367--387, 1996.
  

\bibitem[Meb84]{Meb84}
Zoghman Mebkhout,
Une autre \'{e}quivalence de cat\'{e}gories,
Compositio Math. {\bf51(1)}, 63--88, 1984.

\bibitem[Moc09]{Mochi09}
Takuro Mochizuki,
Good formal structure for meromorphic flat connections on smooth projective surfaces,
In Algebraic analysis and around, Adv. Stud. Pure Math., {\bf54}, 223--253, 2009.

\bibitem[Moc11]{Mochi11}
Takuro Mochizuki,
Wild harmonic bundles and wild pure twistor $\D$-modules,
Ast{\'e}risque, {\bf340}, 2011.

\bibitem[Moc16]{Mochi16}
Takuro Mochizuki,
Curve test for enhanced ind-sheaves and holonomic $\D$-modules, I, II
Annales scientifiques de l'ENS, {\bf 55(3)}, 575--738
, 2022.


\bibitem[Naga62]{Naga62}
Masayoshi Nagata,
Imbedding of an abstract variety in a complete variety, J. Math. Kyoto Univ. 2, 1-10, 1962.

\bibitem[Sai89]{Sai89}
Morihiko Saito,
Induced D-modules and differential complexes,
Bull. Soc. Math. France, 117-3 (1989), 361-387.
\end{thebibliography}
\end{document}